\documentclass[12pt,a4wide,bibtotoc]{article}
\usepackage[a4paper,margin=2.5cm]{geometry}

\usepackage{amssymb}
\usepackage{mathrsfs}
\usepackage{amsmath}
\usepackage{latexsym}

\newtheorem{theorem}{Theorem}[section]

\newtheorem{definition}[theorem]{Definition}
\newtheorem{example}[theorem]{Example}
\newtheorem{remark}[theorem]{Remark}
\newenvironment{proof}[1][Proof]{\begin{trivlist}
\item[\hskip \labelsep {\bfseries #1}]}{\end{trivlist}}
\title{On Moving Frames and Noether's Conservation Laws}
\author{T\^ania M.N.\ Gon\c calves$^a$ \and Elizabeth L.\ Mansfield$^b$}
\date{}

\begin{document}
\maketitle
\begin{center}
\footnotesize{$^a$School of Mathematics, Statistics and Actuarial Science\\ University of Kent\\ Canterbury, CT2 7NF, U.K.\\T.M.N.Goncalves@kent.ac.uk}
\end{center}
\begin{center}
\footnotesize{$^b$School of Mathematics, Statistics and Actuarial Science\\ University of Kent\\ Canterbury, CT2 7NF, U.K.\\E.L.Mansfield@kent.ac.uk}
\end{center}


\begin{abstract}
 Noether's Theorem yields conservation laws for a Lagrangian with a variational symmetry group. The explicit formulae for the laws are well known and the symmetry group is known to act on the linear space generated by the conservation laws. The aim of this paper is to explain the mathematical structure of both the Euler-Lagrange system and the set of conservation laws, in terms of the differential invariants of the group action and a moving frame. For the examples we demonstrate, knowledge of this structure allows the Euler-Lagrange equations to be integrated with relative ease. Our methods take advantage of recent advances in the theory of moving frames by Fels and Olver, and in the symbolic invariant calculus by Hubert. The results here generalise those appearing in Kogan and Olver \cite{KoganOlver} and in Mansfield \cite{Mansfield}. In particular, we show results for high dimensional problems and classify those for the three inequivalent $SL(2)$ actions in the plane. \end{abstract}

\section{Introduction}
In 1918, Emmy Noether proved in \cite{Noether} that for systems derived from a variational principle, conservation laws may be obtained from Lie group actions that leave the functional invariant. Since then, Noether's Theorem has been widely used among applied mathematicians and physicists.

In this paper we present the mathematical structure behind both the Euler-Lagrange equations and the set of conservation laws which come from the application of Noether's Theorem. We show that the new format presented here for the Euler-Lagrange equations and the set of conservation laws can simplify greatly the extremising problem. In particular, we give results for variational problems that are invariant under a Lie symmetry group whose Lie algebra is semisimple, such as $\frak{sl}(n)$, $\frak{so}(n)$, $\frak{su}(n)$, $\frak{sp}(n)$, which are extensively found in physical examples.

Section \ref{Pseccao} of this paper gives a brief introduction to the theoretical foundations of our results: the application of moving frames to actions on jet spaces which yields a ``symbolic invariant calculus'' for differential invariants and their invariant derivatives, the Adjoint action of a Lie group on its Lie algebra, and the Killing form of the Lie algebra. Furthermore, we show how the symbolic invariant calculus can be applied to obtain the Euler-Lagrange equations for variational problems with a Lie group symmetry directly in terms of the invariants. Then in section \ref{novoteorema} we state and demonstrate our main result; Noether's conservation laws can always be written as a divergence of the product of a moving frame with a vector of invariants, where the representation for the moving frame is the inverse of the Adjoint representation of the Lie group on its Lie algebra. The one dimensional case was proved in \cite{Mansfield}, here we extend the result to higher dimensional problems. The main pedagogic example used throughout is the projective action of $SL(2)$ acting on curves in the plane, and on surfaces in $3$-space. In section \ref{ultimo} we show how the integration problem can be reduced for the case of one dimensional Lagrangians that are invariant under a Lie symmetry group whose Lie algebra is semisimple. Finally we classify integration results for Lagrangians that are left unchanged under the three inequivalent $SL(2)$ actions on the plane.

\subsection{Motivating Example}\label{sec:motiv}
Consider the group $SE(2)$, the special (orientation preserving) Euclidean group, acting on the space of curves in the $(x, u(x))$-plane,
\begin{equation}\label{inversase(2)}
\left(\begin{array}{c} x\\ u\end{array}\right)\mapsto 
\left(\begin{array}{c}\widetilde{x}\\ \widetilde{u}\end{array}\right)=\left(\begin{array}{cc} \cos \theta & -\sin\theta\\
\sin\theta & \cos\theta\end{array}\right)\left(\begin{array}{c} x\\ u\end{array}\right) + \left(\begin{array}{c} a\\ b\end{array}\right),
\end{equation}
where $\theta$, $a$ and $b$ are constants that parametrise the group action. 
The Euclidean curvature of a curve $x\mapsto (x, u(x))$, given by
$$\kappa = \frac{u_{xx}}{\left(1+u_x^2\right)^{3/2}},$$
is the lowest order differential invariant, where a differential invariant is an invariant for the prolonged action of a Lie group on a jet-space. All differential invariants for the action (\ref{inversase(2)}) are functions of $\kappa$ and its derivatives with respect to arc length, $s$, where
$$\frac{{\rm d}}{{\rm d}s}=\frac1{\sqrt{1+u_x^2}}\frac{{\rm d}}{{\rm d}x}.$$

Under this action the one dimensional variational problem $\int\,\kappa^2\mathrm{d}s$ has $SE(2)$ as a variational symmetry group. When the  conservation laws arising from the Lie symmetry are calculated using the formulae associated with Noether's Theorem (see \cite{Olver}, \S5.4, and Prop.\ 5.98; the formulae appear complicated but are relatively easily coded), the result  can be arranged in matrix form as $A(x,u,u_x)\boldsymbol{\upsilon}(I)=\mathbf{c}$, where $\boldsymbol{\upsilon}(I)$ is a vector of invariants
and $\mathbf{c}$ are the constants of integration, specifically,
\begin{equation}\label{SE2NTintro}
\begin{pmatrix}
\;x_s & -u_s & 0\;\\
\;u_s & x_s& 0\;\\
\;x u_s - u x_s & u u_s + x x_s & 1\;
\end{pmatrix}
\begin{pmatrix}
-\kappa^2 \\
-2\kappa_s \\
2\kappa 
\end{pmatrix}=
\begin{pmatrix}
c_1\\
c_2\\
c_3 
\end{pmatrix}
\end{equation}

where $u_s=u_x/\sqrt{1+u_x^2}$ and $x_s = 1/\sqrt{1+u_x^2}$, and where this defines $A$ and $\boldsymbol{\upsilon}(I)$. The first conservation law comes from the translation in $x$, the second from the translation in $u$, and the third results from the rotation in the $(x,u)$-plane. The Euler-Lagrange equation for this variational problem was obtained by Euler himself, and is $\kappa_{ss}+\textstyle\frac12\kappa^3=0$, which can be solved in terms of elliptic functions; the extremal curves are also known as Euler's elastica. If one takes a solution for $\kappa$ and inserts it into Equation (\ref{SE2NTintro}) above, then one has three equations for $x$, $x_s$, $u$ and $u_s$ as functions of $s$. Combining these with the defining constraint for $s$, which is $x_s^2+u_s^2=1$, and simplifying, we obtain 
\begin{eqnarray} 
\kappa^4+4\kappa_s^2-(c_1^2+c_2^2)&=&0, \label{firstintEL}  \\  
c_1 u -c_2 x+c_3-2\kappa &=&0, \label{Segunda}\\
u_s(c_1^2+c_2^2)+c_2\kappa^2-2 c_1\kappa_s&=&0.\label{MCpull}
\end{eqnarray}

It can be seen the integration problem is now completely straightforward once $\kappa$ is known. We will show in this paper that results like this are not unusual. 

The matrix $A$ in Equation (\ref{SE2NTintro}) is \emph{equivariant\/}, namely, if one applies the group action to the components then the group action factors out; in this case we have
$$A(\widetilde{x},\widetilde{u},\widetilde{u}_{\widetilde{x}})=R(\theta, a, b)^{-1}A(x,u,u_x),$$ 

where
$$R(\theta, a, b)^{-1}=
\begin{pmatrix}
\cos\theta & \sin\theta& 0\\ 
-\sin\theta&\cos\theta & 0\\ 
b& -a&1
\end{pmatrix}.$$

The matrix $R(\theta,a,b)$ is a representation of $SE(2)$. Indeed, the group product in parameter space is given by
$$(\theta, a,b)\cdot (\phi, \alpha,\beta)=(\theta + \phi, a+\alpha\cos\theta -\beta\sin\theta,b+\alpha\sin\theta+\beta\cos\theta),$$

and it is simple to check that
$$R(\theta,a,b)R(\phi,\alpha,\beta)=R\left((\theta, a,b)\cdot (\phi, \alpha,\beta)\right).$$

In fact, the representation is well-known as the so-called Adjoint representation, see \S 3.3 of \cite{Mansfield}. The map $A$ is thus an example of a {\em moving frame\/}, which is an equivariant map from the space $M$ on which a Lie group $G$ acts, to $G$.

The example of $SE(2)$ invariant Lagrangians with the independent variable being Euclidean arc length was first carried out in \cite{KoganOlver} and is also fully explored in \cite{Mansfield}.

\section{Moving Frames, the Adjoint Action and the Invariant Calculus of Variations}\label{Pseccao}
In this section, we will give a brief description of the concepts needed to explain our results, namely moving frames following the development in \cite{FelsOlver} (and also \cite{Mansfield}), the Adjoint action of a Lie group and the Killing form on its Lie algebra, and the symbolic invariant calculus. We will use the results of the pedagogical examples in the following sections.

A smooth group action on a smooth space induces an action on the set of smooth curves and surface elements  in that space including their higher order derivatives in the relevant jet bundle, the so-called prolonged curves and surfaces. In this paper, the set $M$ on which $G$ acts consists of these prolonged curves and surfaces.

\subsection{Moving Frames}\label{moving}
 A \emph{group action} of $G$ on $M$  is a map $G\times M\rightarrow M$, written as $(g,z)\mapsto g\cdot z$, which satisfies either $g\cdot (h\cdot z)=(gh)\cdot z$, called a \emph{left action}, or $g\cdot (h\cdot z)=(hg)\cdot z$, called a \emph{right action}. We will also write $g\cdot z$ as $\widetilde{z}$ to ease the exposition in places.

We assume that $G$ is a Lie group and that the action is smooth. Further, we assume the action is {\em free\/} and {\em regular\/} in some domain $\mathcal{U}\subset M$, which means, in effect, that 
 
\begin{enumerate}
\item
the intersection of the orbits with  $\mathcal{U}$ have the dimension of the group $G$ and further foliate $\mathcal{U}$;
\item
there exists a surface $\mathcal{K}\subset \mathcal{U}$ that intersects the orbits of $\mathcal{U}$ transversally, and the intersection of an orbit of $\mathcal{U}$ with $\mathcal{K}$ is a single point. This surface $\mathcal{K}$ is known as the \emph{cross-section} and has dimension equal to $\mbox{dim}(M)-\mbox{dim}(G)$;
\item
if we let $\mathcal{O}(z)$ denote the orbit through $z$, then the element $h \in G$ that takes $z\in \mathcal{U}$ to $\{k\}=\mathcal{O}(z)\cap \mathcal{K}$ is unique.
\end{enumerate}

Under these conditions,  an equivariant map $\rho:\mathcal{U}\rightarrow G$ can be defined. Such a map is called a {\em moving frame\/} on $\mathcal{U}$. Specifically, we can define the map $\rho:\mathcal{U} \rightarrow G$ to be the unique element in $G$ which satisfies
$$\rho(z)\cdot z=k,\qquad \{k\}=\mathcal{O}(z)\cap \mathcal{K}.$$
We say $\rho$ is the \emph{right moving frame} relative to the cross-section $\mathcal{K}$. By construction, we have for a left action that $\rho(g\cdot z)=\rho(z)g^{-1}$, and for a right action that $\rho(g\cdot z)=g^{-1}\rho(z)$, so that $\rho$ is indeed equivariant. The cross-section $\mathcal{K}$ is not unique, and is usually selected to simplify the calculations for a given application. In practice, the procedure to find the frame is as follows:

\begin{enumerate}
\item
define the cross-section $\mathcal{K}$ to be the locus of the set of equations $\psi_i(z)=0$, for $i=1,...,r$, where $r$ is the dimension of the group $G$;
\item
find the group element in $G$ which maps $z$ to $k \in \mathcal{K}$ by solving the \emph{normalisation equations},
$$\psi_i(\widetilde{z})=\psi_i(g\cdot z)=0,\qquad i=1,...,r.$$
\end{enumerate}

Hence, the frame $\rho$ satisfies $\psi_i(\rho(z)\cdot z)=0$, $i=1,...,r$.

\begin{example}\label{SL2}
Consider the group $SL(2)$ acting projectively on the plane as follows
$$\widetilde{x}=g\cdot x=x,\qquad \widetilde{u}=g\cdot u=\displaystyle{\frac{au+b}{cu+d}},$$
where 

\begin{equation}\label{sltwoelt} g=
\begin{pmatrix}
a & b \\
c & d
\end{pmatrix}, \qquad ad-bc=1.\end{equation} 

The induced actions on $u_x$ and $u_{xx}$, defined to be that obtained using the chain rule, are respectively
$$g\cdot u_x=\displaystyle{\widetilde{u_x}=\widetilde{u}_{\widetilde{x}}=\frac{u_x}{(cu+d)^2}},$$
$$g\cdot u_{xx}=\displaystyle{\widetilde{u_{xx}}=\widetilde{u}_{\widetilde{x}\widetilde{x}}=\frac{u_{xx}(cu+d)-2cu_x^2}{(cu+d)^3}}.$$
If we take $M$ to be the space with coordinates $(x,u,u_x,u_{xx},u_{xxx},...)$, then the action is locally free near the identity of $SL(2)$ and regular away from the coordinate plane $u_x=0$. Hence we consider $u_x\ne0$, and take the normalisation equations to be $\widetilde{u}=0$, $\widetilde{u_x}=1$, and $\widetilde{u_{xx}}=0$, to obtain
\begin{equation}\label{ProjSL2frame}
\displaystyle{a=\frac{1}{\sqrt{u_x}}},\qquad \displaystyle{b=-\frac{u}{\sqrt{u_x}}},\, and\quad\displaystyle{c=\frac{u_{xx}}{2u_x^{3/2}}}
\end{equation}
as the frame in parametric form, or in matrix form, substituting for $a$, $b$ and $c$ into (\ref{sltwoelt}),
$$\rho(u,u_x,u_{xx})=\left(\begin{array}{cc}
\displaystyle{\frac{1}{\sqrt{u_x}}} &\displaystyle{-\frac{u}{\sqrt{u_x}}}\\[10pt]
\displaystyle{\frac{u_{xx}}{2u_x^{3/2}}} &\displaystyle{\frac{2u_x^2-uu_{xx}}{2u_x^{3/2}}}\end{array}\right).$$
The square root indicates that the domain of the frame is restricted, and the choice of root is such as to
 ensure that $\rho$ is the identity element on the cross-section $\mathcal{K}$. If $x$, $u$ are considered to be real and $u_x<0$ for the application at hand, one can use the frame equation $\widetilde{u_x}=-1$ instead.
\end{example}

\begin{theorem}
 Let $\rho$ be a right moving frame. Then the quantity $I(z)=\rho(z)\cdot z$ is an invariant of the group action (see \cite{FelsOlver}). 
\end{theorem}

If $z=(z_1,...,z_n)$, i.e. $z$ is given in coordinates, and the normalisation equations are $\widetilde{z_i}=c_i$ for $i=1,...,r$, where $r$ is the dimension of the group, then 
$$\rho(z)\cdot z=(c_1,...,c_r,I(z_{r+1}),...,I(z_n)),$$
where
$$I(z_k)=g\cdot z|_{g=\rho(z)},\; \mathrm{for}\;k=r+1,...,n.$$
In this paper we are interested in Lie group actions on jet bundles. We denote the independent variables as $\mathbf{x}=(x_1, x_2, \dots, x_p)$, and the dependent variables as $\mathbf{u}=(u^1, \dots, u^q)$. We denote the derivative terms as
$$ u^{\alpha}_K = \frac{\partial^{|K|}}{\partial x_1^{k_1}\cdots \partial x_p^{k_p}} u^{\alpha}= \partial_K u^{\alpha}$$
where this defines $\partial_K$, with $K$ being  a multi-index of differentiation, $K=(k_1, \dots, k_p)$ and $|K|=k_1+\dots + k_p$. Then coordinates on the $n$-th jet bundle $J^n(\mathbf{x},\mathbf{u})$ are the $x_i$, the $u^{\alpha}$, and the $u^{\alpha}_K$, where $|K|\le n$. Thus, the operator $\partial/\partial x_i$ extends on this space to the \emph{total differentiation operator}
$$D_i=\frac{D}{Dx_i}=\frac{\partial}{\partial x_i}+\sum_{\alpha=1}^q\sum_K u^\alpha_{Ki}\frac{\partial}{\partial u^\alpha_K}.$$
We denote the invariantised jet bundle coordinates as
\begin{equation}\label{notainvariant}
J_i=I(x_i)=\widetilde{x_i}|_{g=\rho(z)},\qquad I^{\alpha}_K=I(u^{\alpha}_K)=\widetilde{u^{\alpha}_K}|_{g=\rho(z)}.
\end{equation}
These  are also known as the \emph{normalised differential invariants}.

\begin{example}\label{exone}
Consider the action of the $SL(2)$ group on the plane, as in Example \ref{SL2}. We have 
$$\begin{array}{rcl}
g\cdot z|_{g=\rho(z)} & = & (\widetilde{x},\widetilde{u},\widetilde{u_x},\widetilde{u_{xx}},\widetilde{u_{xxx}})|_{g=\rho(z)}\\[10pt]
& = &\left(I(x), I^u, I^u_1, I^u_{11}, I^u_{111}\right)\\[10pt]
& = & \displaystyle{\left(x,0,1,0,\frac{u_{xxx}}{u_x}-\frac{3}{2}\frac{u_{xx}^2}{u_x^2}\right)}.
\end{array}$$

The last component is the well-known $SL(2)$ invariant known as the Schwarzian derivative of $u$, often denoted as $\{u;x\}$ . The second, third and fourth components correspond to the normalisation equations $\widetilde{u}=0$, $\widetilde{u_x}=1$, and $\widetilde{u_{xx}}=0$ respectively. Continuing, one could obtain $I^u_{1111}=\left(g\cdot u_{xxxx}\right)\big\vert_{g=\rho}$ and so on. In fact $I^u_{1111}=\{u;x\}_x$, and all the higher order invariants can be obtained in terms of $\{u;x\}$ and its derivatives. 
\end{example}

\begin{theorem}
(\textbf{Replacement Theorem} \cite{FelsOlverII})
 If $f(z)$ is an invariant, then 
$$f(z)=f(I(z)).$$
\end{theorem}

{\bf Example \ref{exone} (cont.)}
\emph{By applying the normalisation equations to the Schwarzian we obtain
$$\frac{u_{xxx}}{u_x}-\frac{3}{2}\frac{u_{xx}^2}{u_x^2}=\frac{I^u_{111}}{I^u_1}-\frac{3}{2}\frac{I^u_{11}}{I^u_1}=I^u_{111},$$
confirming the result above.}
\vspace{0.5cm}

The Replacement Theorem can be used to express historically known invariants in terms of the $I^{\alpha}_K$ invariants even when the normalisation equations cannot be solved for the frame. 

For the pedagogic examples used in this paper, we are able to solve normalisation equations for the frame and explicitly calculate the $I^\alpha_K$. Perhaps the most significant outcome arising from the seminal paper \cite{FelsOlver} is that a \emph{symbolic\/} invariant calculus for the $I^{\alpha}_K$ can be constructed from the normalisation equations alone, that is, \emph{without\/} knowing the frame explicitly. This symbolic calculus was formulated rigorously by Hubert (\cite{HubertAA, HubertAC, HubertAD, HubertA, HubertB}), and a ``working mathematician's guide'' appears in \cite{Mansfield}. Simply put, we can differentiate the invariants $I^{\alpha}_K$ symbolically in terms of the $I^\alpha_K$ and hence calculate the differential relations that they satisfy using symbolic computation software \cite{AIDA}.

The \emph{invariant differential operators} are obtained in a way analogous to that of the normalised differential invariants.

\begin{definition}
A distinguished set $\{\mathcal{D}_i \, |\, i=1, \dots p\}$ of invariant differential operators is obtained by evaluating the transformed total differential operators on the frame, i.e.
$$\mathcal{D}_i=\widetilde{D_i}|_{g=\rho(z)},$$
where $\widetilde{D_i}$ is defined as follows
$$\widetilde{D_i}=\frac{D}{D\widetilde{x_i}}=\sum_{j=1}^p(\widetilde{D}x)_{ij}D_i.$$
Here $(\widetilde{D}x)_{ij}=((D\widetilde{x})^{-1})_{ij}$.
\end{definition}

The invariant differential operators $\mathcal{D}_i$ map differential invariants to differential invariants.

 We know that
$$\frac{\partial}{\partial x^i} u^{\alpha}_K=u^{\alpha}_{Ki},$$
but the same is not true once we invariantise; $\mathcal{D}_i I^{\alpha}_K \neq I^{\alpha}_{Ki}$, and we have

\begin{equation}\label{invdiffnoncom}
\mathcal{D}_i I^{\alpha}_K=I^{\alpha}_{Ki}+M^{\alpha}_{Ki},\end{equation}

where $M^{\alpha}_{Ki}$ is known as the {\em error term\/}. Equation (\ref{invdiffnoncom}) indicates that the processes of differentiation and invariantisation do not commute. The error terms may be calculated from knowledge of the normalisation equations and the infinitesimal action alone, that is, without solving for the frame explicitly (\cite{Mansfield} \S 5.1.1), and symbolic software that implements the formulae have been written (\cite{AIDA} amongst others).

Consider the two generating differential invariants $I^\alpha_J$ and $I^\alpha_L$ and let $JK=LM$
so that $I^\alpha_{JK}=I^\alpha_{LM}$. This implies that
\begin{equation}\label{SYZ}
\mathcal{D}_KI^\alpha_J-M^\alpha_{JK}=\mathcal{D}_MI^\alpha_L-M^\alpha_{LM}.
\end{equation}
Equations such as (\ref{SYZ}) are called \emph{syzygies} or \emph{differential identities}.

\vspace{0.5cm}
 {\bf Example \ref{exone} (cont.)} \emph{If we now set $u=u(x,\tau)$,
and take the same normalisation equations as before, we obtain
$$ \widetilde{u_{\tau}} |_{g=\rho(z)} = I^u_2 =\frac{u_{\tau}}{u_x}.$$
Further, since both $x$ and $\tau$ are invariant, $\mathcal{D}_\tau=\partial/\partial {\tau}$ and $\mathcal{D}_x=\partial/\partial x$. Next,
 $$\mathcal{D}_\tau I^u_{111}=I^u_{1112}-I^u_{12}I^u_{111},\qquad \mathcal{D}_x^3 I^u_2=I^u_{1112}-3I^u_{12}I^u_{111}-I^u_{1111}I^u_{2}$$
 so that eliminating the $I^u_{1112}$ term, and noting that $\mathcal{D}_x I^u_2=I^u_{12}$ and $\mathcal{D}_x I^u_{111}=I^u_{1111}$ gives a syzygy between $I^u_2$ and $I^u_{111}$. The syzygy is
\begin{equation}\label{syzygysigma}
\mathcal{D}_\tau \sigma=(\mathcal{D}_x^3+2\sigma\mathcal{D}_x+\sigma_x)I^u_2,
\end{equation}
where $\sigma=I^u_{111}=\{u;x\}$, which can be verified directly. In this case, it can be shown that the invariants $I^u_2$ and $I^u_{111}$ generate the set of all differential invariants under invariant differentiation and functional composition.}
\vspace{0.5cm}

Equation (\ref{syzygysigma}) is an example of the presentation of the syzygies we will need to obtain our results. Theorems concerning the finite generation of the algebra of invariants, and their related syzygies have been given by Hubert (\cite{HubertAC, HubertAD}). Syzygies given in the form of Equation (\ref{syzformHdef}), needed for our calculations which follow,
 will hold for a wide class of group actions and their moving frames.

\subsection{The Adjoint Action and the Killing Form}\label{adjointaction}
In this section we briefly review the Adjoint action and the Killing form for a Lie group. The calculations we show will be needed in Section \ref{ultimo}.

Suppose the Lie group $G$ acts on the smooth space $M$ with local coordinates $(z_1, \dots, z_n)$. We denote by $\mathcal{X}(M)$ the space of vector fields on $M$. By an abuse of notation, for any $g\in G$ we denote the smooth map $z\mapsto g\cdot z$ also by $g:M\rightarrow M$.

\begin{definition}
The action $Ad$ of $G$ on $\mathcal{X}(M)$ is 
\begin{equation}\label{Addef}
(g,\mathbf{v})\mapsto Ad_g(\mathbf{v}),\qquad Ad_g(\mathbf{v})(z)=Tg^{-1}\mathbf{v}(g\cdot z),
\end{equation}
where $Tg: TM\rightarrow TM$ is the tangent map of $g: M\rightarrow M$.
\end{definition}

In coordinates, if
$$\mathbf{v}=\sum_jf_j(z)\frac{\partial}{\partial z_j}= \mathbf{f}^T\boldsymbol{\nabla},$$ then
\begin{displaymath}
Ad_g(\mathbf{v})
 =\displaystyle{\left(\left(\frac{\partial\widetilde{z}}{\partial z}\right)^{-1}\mathbf{f}(\widetilde{z})\right)^T}\boldsymbol{\nabla},
\end{displaymath}
where $(\partial\widetilde{z}/\partial z)$ is the Jacobian of the map $z\mapsto g\cdot z=\widetilde{z}$.

It can be seen that $Ad_g$ is a linear map on $\mathcal{X}(M)$, and further that
$Ad_g\circ Ad_h=Ad_{gh}$ by the chain rule.

Given a smooth group action of $G$ on $M$ where $\dim G=r$, there is an $r$ dimensional vector subspace $\mathcal{X}_G(M)\subset \mathcal{X}(M)$, the so-called infinitesimal vector fields of the group action, which is a representation of the Lie algebra $\mathfrak{g}$ of G, obtained as follows. We take the view that $\mathfrak{g}$ is the {\em tangent space\/} $T_eG$ of $G$ at its identity element $e$, and that this space is modelled by smooth paths $\gamma:[-\epsilon,\epsilon]\rightarrow G$, $\epsilon >0$, such that $\gamma(0)=e$, where paths are considered to be equivalent if their derivatives at $e$ are equal (see for example \cite{Hirsch}). Such a path generates a smooth path through every element $z\in M$ given by $t\mapsto\gamma(t)\cdot z$, and then the derivative of each path at $t=0$ yields a vector field on $M$. The set of such vector fields is $\mathcal{X}_G(M)$. A standard result is that this is a linear space and that a basis of $T_eG$ yields a basis of $\mathcal{X}_G(M)$. It can be shown from the definition of $Ad$, Equation (\ref{Addef}), that $Ad_g:\mathcal{X}_G(M)\rightarrow \mathcal{X}_G(M)$ by noting that the map $\gamma\mapsto g^{-1}\gamma g$ takes $T_eG$ to itself. 
 
 \begin{example}\label{SL2cont}
Consider the $SL(2)$ action as in Example \ref{SL2}. Paths at the identity of $G=SL(2)$ yield  paths
$$t\mapsto \left(x, \frac{a(t)u+b(t)}{c(t)u + \left(1+b(t)c(t)\right)/a(t)}\right)$$
where $a(0)=1$, $b(0)=c(0)=0$ and where $a'(0)=\alpha$, $b'(0)=\beta$, and $c'(0)=\gamma$ are independent constants. Differentiating at $t=0$ yields the three dimensional vector space of infinitesimal vector fields,  with basis

\begin{equation}\label{Xsl2basis}
\mathbf{v}_1= 2u\partial_u,\qquad \mathbf{v}_2= \partial_u,\qquad \mathbf{v}_3= -u^2\partial_u
\end{equation}

and generic element

\begin{equation}\label{sl2genericvf}\mathbf{v}=(\alpha(2u)+\beta+\gamma(-u^2))\partial_u.\end{equation}

Extending the action to the $(x,u,u_x,u_{xx},\dots)$-space via the chain rule leads to the ``prolongation" of vector fields, for example $\mathbf{v_3}$ prolongs to (\cite{Olver}, \S2.3)
$$\mathbf{v_3} = -u^2\partial_u -2uu_x\partial_{u_x} -(2u_x^2 + 2 uu_{xx})\partial_{u_{xx}} -\cdots$$
\end{example}

Given a basis $\mathbf{v}_i$ of $\mathcal{X}_G(M)$, $i=1,...,r$, 
$$Ad_g\left(\sum_i\alpha_i\mathbf{v}_i\right)=\sum_i\alpha_iAd_g(\mathbf{v}_i)=\sum_{i,\,j}\alpha_i\mathcal{A}d(g)^i_j\mathbf{v}_j,$$
for some $r\times r$ matrix $\mathcal{A}d(g)$. In practice, it can be easier to calculate the induced action on the coefficients $\alpha_i$, 
$$\sum_{i,\,j}\alpha_i\mathcal{A}d(g)^i_j\mathbf{v}_j=\sum_i \widetilde{\alpha_i}\mathbf{v}_i,$$
so that writing $\boldsymbol{\alpha}$ as a column vector, $\widetilde{\boldsymbol{\alpha}}=\mathcal{A}d(g)^T\boldsymbol{\alpha}$. If a basis of infinitesimal vector fields $\{\mathbf{v}_i\}$ is given on a space with coordinates $(z_1, z_2, \dots, z_n)$, with $\mathbf{v}_i=\sum\zeta^i_r\partial_{z_r}$, we define the matrix of infinitesimals $\Omega^\alpha(z)$ to be

\begin{equation}\label{infmatgen}
\Omega^\alpha(z) = (\Omega^\alpha_{ir})=(\zeta^i_r).\end{equation} 

In terms of the matrix of infinitesimals, the matrix $\mathcal{A}d(g)$ satisfies

\begin{equation}\label{AdactMxform}
{\cal A}d(g)\Omega(z)=\Omega(\widetilde{z})\left(\frac{\partial \widetilde{z}}{\partial z}
\right)^{-T}.
\end{equation}

In the following example we calculate $\mathcal{A}d(g)$ and verify Equation (\ref{AdactMxform}). 

\vspace{0.5cm}
{\bf Example \ref{SL2cont} (cont.)}\ 
\emph{To find $\mathcal{A}d(g)$, we calculate the Adjoint action of $g\in SL(2)$ on the generic infinitesimal vector field given in (\ref{sl2genericvf}). We obtain
$$\begin{array}{rcl}
Ad_g(\mathbf{v})(z) & = & \displaystyle{(\alpha(2\widetilde{u})+\beta+\gamma(-\widetilde{u}^2))\frac{\partial}{\partial \widetilde{u}}}\\[10pt]
& = & \displaystyle{(\widetilde{\alpha}(2u)+\widetilde{\beta}+\widetilde{\gamma}(-u^2))\frac{\partial}{\partial u}},
\end{array}$$
so that
\begin{equation}\label{Admatrixrep}\begin{pmatrix}
\widetilde{\alpha}\\
\widetilde{\beta}\\
\widetilde{\gamma}\end{pmatrix} = \mathcal{A}d(g)^T\begin{pmatrix}
\alpha\\
\beta\\
\gamma\end{pmatrix} = 
\begin{pmatrix}
 ad+bc & cd & -ab\\  
 2bd & d^2 & -b^2\\  
 -2ac & -c^2 & a^2 \end{pmatrix} 
\begin{pmatrix}
\alpha\\
\beta\\
\gamma
\end{pmatrix}.\end{equation}
On $(u,u_x)$-space, the matrix of infinitesimals is
$$\Omega^u(z)=\bordermatrix{ & u & u_x\cr a& 2u & 2 u_x \cr b& 1&0\cr c& -u^2 & -2uu_x}$$
and with $z=(u,u_x)$ we have 
$$\frac{\mbox{D}\widetilde{z}}{\mbox{D}z}=
\begin{pmatrix} \displaystyle{\frac{\partial \widetilde{u}}{\partial u}} & \displaystyle{\frac{\partial \widetilde{u}}{\partial u_x}}\\[10pt]
\displaystyle{\frac{\partial \widetilde{u_x}}{\partial u}} & \displaystyle{\frac{\partial \widetilde{u_x}}{\partial u_x}}\end{pmatrix}
=\begin{pmatrix} \displaystyle{\frac1{(cu+d)^2}} & 0 \\[10pt] \displaystyle{\frac{-2c u_x}{(cu+d)^3}} & \displaystyle{\frac1{(cu+d)^2}}\end{pmatrix}.$$
Equation (\ref{AdactMxform}) is easily verified.}

\begin{remark} \emph{There are several reasons  for considering the Adjoint action of a Lie group $G$, not on its matrix Lie algebra but on the  representation of  the Lie algebra, $\mathcal{X}_G(M)$. To begin with, Lie symmetries of variational problems are found using symbolic software which return the vector field representation of the Lie algebra; it is the flows of these fields that generate $G$ so that the (local) group action is found by integrating the infinitesimal vector fields. Even more importantly, it is the infinitesimal vector fields that appear in the derivation of the formulae for Noether's Theorem.}
\end{remark}

If $\mathbf{v}\in \mathcal{X}_G(M)$, then there is a linear map, called the adjoint map of $\mathbf{v} $,
$$\mbox{ad}_{\mathbf{v}}:\mathcal{X}_G(M)\rightarrow
\mathcal{X}_G(M),\qquad \mbox{ad}_{\mathbf{v}}(\mathbf{w}) = [\mathbf{v},\mathbf{w}]$$
where $[\, , \, ]$ is the standard bracket of vector fields. A standard calculation yields
$$ \mbox{ad}_{{A}d_g(\mathbf{v})} = {A}d_g \circ \mbox{ad}_{\mathbf{v}} \circ {A}d_g^{-1}.$$
If one takes a basis $  \mathbf{v}_1$, \dots, $ \mathbf{v}_r$ of $\mathcal{X}_G(M)$, where $r=\mbox{dim}(G)$, then an $r\times r$ matrix representation of $\mbox{ad}_{\mathbf{v}}$ can be obtained to which we give the same name. The bilinear \emph{Killing form\/} $B$ on $\mathcal{X}_G(M) $ is defined to be
$$B(\mathbf{v}, \mathbf{w})=\mbox{trace}\left(\mbox{ad}_{\mathbf{v}}\mbox{ad}_{\mathbf{v}}\right)$$
and this form is then overtly $\mathcal{A}d_g$ invariant. In terms of the matrix
$$\mathbf{B}=(B_{ij}),\qquad B_{ij}=B(\mathbf{v}_i, \mathbf{v}_j),$$ the $\mathcal{A}d_g$ invariance takes the form

\begin{equation}\label{BAdinveq} \mathcal{A}d_g\, \mathbf{B}\, \mathcal{A}d_g^{\,T}=\mathbf{B}.\end{equation}

\vspace{0.5cm}
{\bf Example \ref{SL2cont} (cont.)}
 \emph{We continue with the projective $SL(2)$ action described above. The basis of $\mathcal{X}_{SL(2)}(M)$ is given in (\ref{Xsl2basis}) and setting $\mathbf{v}=\alpha\mathbf{v}_1+  \beta\mathbf{v}_2+   \gamma\mathbf{v}_3$ we have relative to that basis that
$$\mathrm{ad}_{\mathbf{v}} =\begin{pmatrix} 0& 2\beta & -2\gamma \\ \gamma & -2\alpha & 0\\ -\beta & 0 & 2\alpha\end{pmatrix}$$
and hence the Killing form is 
\begin{equation}\label{killingmatrix}
\mathbf{B}=\begin{pmatrix}8&0&0\\0&0&4\\0&4&0\end{pmatrix}.
\end{equation}
The ${A}d$-invariance of $\mathbf{B}$ using $\mathcal{A}d(g)$ from Equation (\ref{Admatrixrep}) is easily verified.}

\subsection{The Invariant Calculus of Variations}\label{eulerlagrange}
We assume the independent variables are $\mathbf{x}=(x_1, \dots, x_p)$, the dependent variables are $\mathbf{u}=(u^1, u^2, \dots, u^q)$, and that the Lagrangian is a smooth function of  $\mathbf{x}$, $\mathbf{u}$ and finitely many derivatives of the $u^\alpha$; such a Lagrangian is denoted as $\mathscr{L}[\mathbf{u}]=\int L[\mathbf{u}] \,\mathrm{d}\mathbf{x}$.

Suppose we have a Lagrangian that is invariant under some smooth action of a Lie group $G$. Let $\kappa_j$, $j=1,\dots, N$ be the generating differential invariants of the group action. By the Replacement Theorem, we may assume the $\kappa_j$ are in fact some set of the $I^{\alpha}_K$ and their derivatives with respect to the independent variables, and with respect to a frame $\rho$, on some open domain in the $(x_i, u^{\alpha}, u^{\alpha}_K)$-space. We suppose that the action leaves the $x_i$ invariant, so that the variational problem can be written as $\int\,L[\boldsymbol{\kappa}]\,{\rm d}\mathbf{x}$. This can always be achieved by reparametrisation, and setting the original independent variables to be dependent on the new invariant parameters. Note that if a parameter is assumed to be an arc length then the relevant constraint needs to be inserted with a Lagrange multiplier. Reparametrisation has the additional advantage that the resulting Euler-Lagrange system is a differential system with respect to standard commuting differential operators, so that standard solution methods can be applied.

The Euler-Lagrange equations for such a problem have $G$ as a Lie symmetry, so that they can be expressed as differential equations for the $\kappa_j$ (there can also be trivial non-invariant multipliers which do not affect the solution space and can be discarded). Kogan and Olver \cite{KoganOlver} constructed a variational tricomplex to show how to derive the Euler-Lagrange equations directly in terms of the invariants, bypassing the need to use the standard formulae and then invariantising. Here we show how the invariantised Euler-Lagrange equations can be obtained using calculations  which are close in style to those used to obtain them in the underlying $(\mathbf{x}, \mathbf{u})$ variables. The syzygies discussed in Section \ref{moving} play a central role.
 
If $\mathbf{x}\mapsto (\mathbf{x},\mathbf{u}(\mathbf{x}))$ extremises the functional $\mathscr{L}[\mathbf{u}]$, then for a small perturbation of $\mathbf{u}$ we obtain
\begin{equation}\label{Elderivnoninv}\begin{array}{rcl}
0 & = & \displaystyle{\left.\frac{\mathrm{d}}{\mathrm{d}\varepsilon}\right|_{\varepsilon=0}
\mathscr{L}[\mathbf{u}+\varepsilon \mathbf{v}]}\\[10pt]
& = & \displaystyle{\int \sum_{\alpha=1}^q\left[  \mathsf{E}^{\alpha}(L)v^\alpha+\sum_i\frac{\mbox{D}}{\mbox{D} x_i}\left(\frac{\partial L}{\partial u^\alpha_i}v^\alpha+\cdots\right)\right]}\mathrm{d}
\mathbf{x},
\end{array}\end{equation}
after differentiation under the integral sign and integration by parts, where $\mbox{D}/\mbox{D}x_i$ is the total derivative operator with respect to $x_i$, and
where 
$$\displaystyle{\mathsf{E}^{\alpha}(L)=\sum_K(-1)^{|K|}\frac{\partial^{|K|}}{\partial x_1^{k_1}\cdots \partial x_p^{k_p}}\,\frac{\partial L}{\partial u^\alpha_K},}$$
is the Euler operator with respect to the dependent variable $u^{\alpha}$ acting on $L$. The boundary terms play an important role in the determination of the natural boundary conditions, and also the formulae for Noether's Theorem, in the event the perturbation is given by the group action. We note that the boundary terms are linear in the $v^{\alpha}$ and their derivatives.

In order to obtain the invariantised analogue of ${\left.\frac{\mathrm{d}}{\mathrm{d}\varepsilon}\right|_{\varepsilon=0}\mathscr{L}[\mathbf{u}+\varepsilon \mathbf{v}]}$, where the Lagrangian is given in terms of differential invariants, we first introduce a dummy invariant independent variable, $\tau$. Since both $\tau$ and the $x_i$ are invariant, by construction and hypothesis respectively, we have for all $i$ that
\begin{equation}\nonumber
\displaystyle{\mathcal{D}_\tau=\frac{\mathrm{D}}{\mathrm{D}\tau},\qquad \mathcal{D}_i=\frac{\mathrm{D}}{\mathrm{D} x_i}, \qquad\mathrm{and}\qquad [\mathcal{D}_\tau,\mathcal{D}_i]=0}.
\end{equation}
Furthermore, symbolically,
$$\displaystyle{\left.\frac{\mathrm{d}}{\mathrm{d}\varepsilon}\right|_{\varepsilon=0}\mathscr{L}[\mathbf{u}+\varepsilon \mathbf{v}]=\left.\frac{\mathrm{D}}{\mathrm{D} \tau}\right|_{\mathbf{u}_{\tau}=\mathbf{v}}\mathscr{L}[\mathbf{u}]}.$$
We assume that $L=L[\boldsymbol{\kappa}]$, where $\kappa_j=I(u^{\alpha_j}_{K_j})$, that is, the invariants used to express $L$ are symbolic invariants obtained via a moving frame, as in Equation (\ref{notainvariant}). It turns out it is not necessary to append the syzygies between the $\kappa_j$ as constraints  to $L$;  since we are obtaining the Euler-Lagrange equations with respect to the variables $\mathbf{u}$ and in terms of those variables, the syzygies are identically zero, and hence the syzygies do not contribute. To demonstrate this, in Example \ref{sl2surfWITHLags} we keep the additional syzygy as a constraint with Lagrange multiplier and show that this term disappears in the final result. The introduction of a new independent variable results in $q$ new invariants, $I^{\alpha}_{\tau}= g\cdot u^{\alpha}_{\tau}\big\vert_{g=\rho}$ (as in Equation \ref{notainvariant}), for $\alpha=1, \dots, q$, and a set of syzygies $\mathcal{D}_{\tau} \boldsymbol{\kappa} = \mathcal{H} I(\mathbf{u}_{\tau})$ that is,
\begin{equation}\label{syzformHdef}
\mathcal{D}_{\tau}\left(\begin{array}{c} \kappa_1\\ \kappa_2\\ \vdots \\ \kappa_N\end{array}\right)
=\mathcal{H}\left(\begin{array}{c} I_{\tau}^1\\ I_{\tau}^2\\ \vdots \\ I_{\tau}^q\end{array}\right),
\end{equation}
where $\mathcal{H}$ is a $N\times q$ matrix of operators depending only on the $\mathcal{D}_i$, the $\kappa_j$ and their invariant derivatives.

Mirroring the calculation of $E^{\alpha}(L)$, we have
$$\begin{array}{rcl}
0&=&\displaystyle{\frac{\partial}{\partial \tau}\int L[\boldsymbol{\kappa}]\mathrm{d}\mathbf{x}}\\[10pt]
&=&\displaystyle{\int \left[\sum_{j,K} \frac{\partial L}{\partial \mathcal{D}_K\kappa_j}\mathcal{D}_K\mathcal{D}_{\tau}\kappa_j\right]\mathrm{d}\mathbf{x}}\\[10pt]
&=&\displaystyle{\int \left[\sum_{j,K} (-1)^{|K|}\mathcal{D}_K\frac{\partial L}{\partial \mathcal{D}_K\kappa_j}\mathcal{D}_{\tau}\kappa_j\right]\mathrm{d}\mathbf{x}}+\mbox{B.T.'s}\\[10pt]
&=&\displaystyle{\int \sum_{j,\alpha} \left[ \mathsf{E}^{j}(L)\mathcal{H}_{j,\alpha}I^{\alpha}_{\tau}\right]\mathrm{d}\mathbf{x}}+\mbox{B.T.'s}\\[10pt]
&=&\displaystyle{\int \left[\sum_{j,\alpha}\mathcal{H}^{\ast}_{j,\alpha} \mathsf{E}^{j}(L))I_{\tau}^{\alpha}+\sum_i\frac{\mathrm{D}}{\mathrm{D} x_i}\left(\sum_{J,\alpha}I^{\alpha}_{\tau J}C^{\alpha}_{i,J}\right)\right]\mathrm{d}\mathbf{x}},
\end{array}$$

where ``B.T.'s'' stands for boundary terms, $\mathsf{E}^j$ is the Euler operator corresponding to variations in $\kappa_j$, $\mathcal{H}^\ast_{j,\alpha}$ is the adjoint of $\mathcal{H}_{j,\alpha}$ and $C^{\alpha}_{i,J}$ is the coefficient of $I^\alpha_{\tau J}$ under the $\frac{\mathrm{D}}{\mathrm{D}x_i}$ operator in the B.T.'s. In principle one could add a null divergence to the boundary terms without changing the above equation. However, if one uses standard integration by parts procedures in the calculation of the B.T.'s, one obtains an expression that is linear in the $I^\alpha_{\tau J}$ and from this point of view the $C^\alpha_{i,J}$ are well-defined. By the definition of $I^{\alpha}_{\tau}$ we know that $I^{\alpha}_{\tau}$ contains $u^{\alpha}_\tau$. Thus, from the Fundamental Lemma of Calculus of Variations, the coefficient of $I^\alpha_{\tau}$ must be zero, that is, $\mathsf{E}^{\alpha}(L)=\sum_j \mathcal{H}^*_{j,\alpha} \mathsf{E}^j(L)$ or in matrix form,
\begin{equation}\label{invELeq} 
\mathsf{E}^{\mathbf{u}}(L) = \mathcal{H}^{\ast} \mathsf{E}^{\boldsymbol{\kappa}}(L)
\end{equation}
where $(\mathcal{H}^{\ast})_{a,b}=(\mathcal{H}_{b,a})^{\ast}$. Equation (\ref{invELeq}) represents the invariantised Euler-Lagrange equations. 

This derivation of the invariantised analogue of $\left.\frac{\mathrm{d}}{\mathrm{d}\varepsilon}\right|_{\varepsilon=0}\mathscr{L}[\mathbf{u}+\varepsilon\mathbf{v}]$ can be found in \S 7 of \cite{Mansfield}.
 
\vspace{0.5cm}
 {\bf Example \ref{exone} (cont.)} \emph{We continue with the study of curves in the plane under the projective action of $SL(2)$, with $\sigma=\{u;x\}$, the Schwarzian derivative. Suppose we have the variational problem $\int L[\sigma]\, {\rm d}x$. Introducing the dummy variable $\tau$, with $u=u(x,\tau)$ to effect the variation and recalling that we have 
$$I^u_{\tau}=\frac{u_{\tau}}{u_x},\qquad \mathcal{D}_{\tau}\sigma=\mathcal{H}I^u_{\tau},$$
 where $\mathcal{H}=\mathcal{D}_x^3 + 2\sigma\mathcal{D}_x+\sigma_x$, then the Euler-Lagrange equation of $L$ with respect to $u$ is
 $$\mathsf{E}^u(L)=\mathcal{H}^{\ast}\mathsf{E}^{\sigma}(L).$$
 In this case, $\mathcal{H}^{\ast}=-\mathcal{H}$, so that, for example, if $L[\sigma]=\textstyle\frac12\sigma_x^2$, then $$\mathsf{E}^u(L)=-\left(\mathcal{D}_x^3 + 2\sigma\mathcal{D}_x+\sigma_x\right)\left(-\sigma_{xx}\right),$$ which can be verified directly.}
 
 \begin{example}\label{sl2surfWITHLags} We now consider the projective action of $SL(2)$ on surfaces, $u=u(x,t)$, that is
 $$\widetilde{x}=x,\qquad \widetilde{t}=t,\qquad \widetilde{u}=\frac{au+b}{cu+d},\qquad ad-bc=1.$$
 Take the normalising equations for the frame to be $\widetilde{u}=0$, $\widetilde{u_x}=1$ and $\widetilde{u_{xx}}=0$ as before. The generating invariants are then, as above, $\sigma=\{u;x\}=I^u_{111}$ and $\kappa=u_t/u_x=I^u_2$ and their syzygy is $\mathcal{D}_t\sigma=\left(\mathcal{D}_x^3 + 2\sigma\mathcal{D}_x+\sigma_x\right)\kappa$ as above. Suppose we have a variational problem $\int L[\sigma,\kappa]\,{\rm d}x{\rm d}t$. Introduce the dummy variable $\tau$, and set  $u=u(x,t,\tau)$ to effect the variation. We now have a new invariant, $I^u_{\tau}=u_{\tau}/u_x$ (by the same calculation that yields $I^u_2=\kappa$), and new syzygies are calculated via the method of Equation (\ref{SYZ}), yielding
$$\mathcal{D}_{\tau}\left(\begin{array}{c}\sigma\\ \kappa\end{array}\right)=\left(\begin{array}{c} \mathcal{H}_1 \\ \mathcal{H}_2 \end{array}\right) I^u_{\tau}$$
where $\mathcal{H}_1=\mathcal{H}$, and $\mathcal{H}_2=\mathcal{D}_t -\kappa\mathcal{D}_x + \mathcal{D}_x(\kappa)$. We introduce the syzygy between $\sigma$ and $\kappa$  as a constraint into the Lagrangian with a Lagrange multiplier $\lambda(x,t)$ in order to show what happens. It turns out that $\lambda$ does not appear in the final result; this is expected since we are obtaining the Euler-Lagrange equation with respect to the variable $u$ and in terms of that variable, the syzygy is identically zero. The calculation of the Euler-Lagrange equation of $L$ with respect to $u$ is calculated as follows,
\begin{align*}
 &\mathcal{D}_{\tau}\int [L[\sigma,\kappa]+\lambda\left(\mathcal{D}_t\sigma-\mathcal{H}\kappa\right)]\,{\rm d}x{\rm d}t\\  
& = \int\left[\left(\sum_{K}\frac{\partial L}{\partial \mathcal{D}_K\sigma}\mathcal{D}_K\right)\mathcal{D}_{\tau}\sigma  +  \left(\sum_{K}\frac{\partial L}{\partial \mathcal{D}_K\kappa}\mathcal{D}_K\right)\mathcal{D}_{\tau}\kappa +\lambda\mathcal{D}_{\tau}\left(\mathcal{D}_t\sigma-\mathcal{H}\kappa\right)\right]\, {\rm d}x{\rm d}t \\
 &= \int \left[\left( \mathsf{E}^{\sigma}(L) -\lambda_t -\lambda\kappa_x +\lambda_x\kappa\right)\mathcal{D}_{\tau}\sigma +\left( E^{\kappa}(L) +\lambda\sigma_x+2\lambda_x\sigma+\lambda_{xxx}\right)\mathcal{D}_{\tau}\kappa \right] {\rm d}x{\rm d}t \\
 &= \int \left[\left( \mathsf{E}^{\sigma}(L) -\lambda_t -\lambda\kappa_x +\lambda_x\kappa\right)\mathcal{H}_{1}I^u_{\tau} +\left( \mathsf{E}^{\kappa}(L) +\lambda\sigma_x+2\lambda_x\sigma+\lambda_{xxx}\right)\mathcal{H}_{2}I^u_{\tau} \right]\, {\rm d}x{\rm d}t \\
 &=\int \left[\mathcal{H}_1^{\ast}\left(\mathsf{E}^{\sigma}(L) -\lambda_t -\lambda\kappa_x +\lambda_x\kappa\right) +\mathcal{H}_2^{\ast}\left(E^{\kappa}(L) +\lambda\sigma_x+2\lambda_x\sigma+\lambda_{xxx}\right)\right]I^u_{\tau} \,{\rm d}x{\rm d}t. 
 \end{align*}
 Suppressing the boundary terms, we obtain 
 $$\mathsf{E}^u(L)=\mathcal{H}_1^{\ast}\left(\mathsf{E}^{\sigma}(L) -\lambda_t -\lambda\kappa_x +\lambda_x\kappa\right) +\mathcal{H}_2^{\ast}\left(\mathsf{E}^{\kappa}(L) +\lambda\sigma_x+2\lambda_x\sigma+\lambda_{xxx}\right).$$
 In fact, the terms involving $\lambda$ greatly simplify, to be $2\lambda_x(\mathcal{H}\kappa-\sigma_t) +\lambda(\mathcal{H}\kappa-\sigma_t)_x$ which is identically zero by virtue of the syzygy. Hence we obtain finally
$$\mathsf{E}^u(L)=\mathcal{H}_1^{\ast}\mathsf{E}^{\sigma}(L)+\mathcal{H}_2^{\ast} \mathsf{E}^{\kappa}(L)=-\left(\mathcal{D}_x^3 + 2\sigma\mathcal{D}_x+\sigma_x\right)\mathsf{E}^{\sigma}(L) + \left(-\mathcal{D}_t +\kappa\mathcal{D}_x+2\kappa_x\right)\mathsf{E}^{\kappa}(L).$$
 \end{example}
  
\section{Structure of Noether's Conservation Laws}\label{novoteorema}
Consider Equation (\ref{Elderivnoninv}) where the variation comes from a group action $u^{\alpha}\mapsto g\cdot u^{\alpha}$ and which leaves the independent variables invariant. For any path $g(t)\subset G$ with $g(0)=e$ we will have $v^{\alpha}=\mathrm{d}/\mathrm{d}t \big\vert_{t=0} g(t)\cdot u^{\alpha}$, that is, the $u^\alpha$ component of an infinitesimal vector field. In this case, we have by the invariance of $L$ that
$$0=\sum_{\alpha} v^{\alpha} E^{\alpha}(L) + \sum_i \frac{\mathrm{D}}{\mathrm{D}x_i} P_i$$
so that 
$$\sum_i \frac{\mathrm{D}}{\mathrm{D}x_i} P_i = 0$$
on solutions of the Euler-Lagrange system. These conservation laws could well be identically zero, but we do not address this case in this paper. This result is Noether's Theorem and we will obtain a conservation law for every infinitesimal vector field. It can be seen since the expressions for the $P_i$ are linear in the $v^{\alpha}$ and their derivatives that we need consider only a basis of infinitesimal vector fields. The formulae for the $P_i$ are well known \cite{Olver}. In the one dimensional problem, we obtain $r=\mbox{dim}(G)$ first integrals of the Euler-Lagrange equations.

\vspace{0.5cm}
{\bf Example \ref{exone} (cont.)}\emph{
Consider the $SL(2)$ group action as in Example \ref{SL2} and consider the Lagrangian
$$L(\sigma,\sigma_x,\sigma_{xx},...)\mathrm{d}x$$
where $$\sigma=\displaystyle{\frac{u_{xxx}}{u_x}-\frac{3}{2}\frac{u_{xx}^2}{u_x^2}=I^u_{111}.}$$
The group $SL(2)$ is a three parameter group and so there are three conservation laws. Calculating these according to the known formulae, and writing third order and higher derivatives of $u$ in terms of $\sigma$, these three laws are
$$\underbrace{\left(\begin{array}{ccc}
\displaystyle{1-\frac{uu_{xx}}{u_x^2}} &\kern8pt \displaystyle{\frac{2u}{u_x}} &\kern8pt \displaystyle{\frac{u_{xx}}{u_x}-\frac{uu_{xx}^2}{2u_x^3}}\\[10pt]
\displaystyle{-\frac{u_{xx}}{2u_x^2}} &\kern8pt \displaystyle{\frac{1}{u_x}} &\kern8pt \displaystyle{-\frac{u_{xx}^2}{4u_x^3}}\\[10pt]
\displaystyle{-u+\frac{u^2u_{xx}}{2u_x^2}} &\kern8pt \displaystyle{-\frac{u^2}{u_x}} &\kern8pt \displaystyle{u_x-\frac{uu_{xx}}{u_x}+\frac{u^2u_{xx}^2}{4u_x^3}}
\end{array}\right)}_{\displaystyle{\mathcal{A}d(\rho)^{-1}}}\underbrace{\left(\begin{array}{c}
\displaystyle{-2\frac{\mathrm{d}}{\mathrm{d}x}\mathsf{E}^\sigma(L)}\\[10pt]
\displaystyle{\sigma\mathsf{E}^\sigma(L)+\frac{\mathrm{d}^2}{\mathrm{d}x^2}\mathsf{E}^\sigma(L)}\\[10pt]
-2\mathsf{E}^\sigma(L)
\end{array}\right)}_{\displaystyle{\boldsymbol{\upsilon}(I)}}=\left(\begin{array}{c}
c_1\\[10pt]
c_2\\[10pt]
c_3
\end{array}\right)$$
where the matrix on the left equals $\mathcal{A}d(\rho)^{-1}$, the inverse of the representation of the $SL(2)$ Adjoint action on the vector fields
obtained in Equation (\ref{Admatrixrep}) and evaluated on the moving frame obtained in Equation (\ref{ProjSL2frame}), and $\boldsymbol{\upsilon}(I)$ is a vector of invariants.}

\vspace{0.5cm}
The following theorem generalises the result appearing in \cite{Mansfield}.

\begin{theorem}\label{tmel}
Let $\int L(\kappa_1,\kappa_2,...)\mathrm{d}\mathbf{x}$ be invariant under $G\times M\rightarrow M$, where $M=J^n(X\times U)$, with generating invariants $\kappa_j$, for $j=1,...,N$, and let $g\cdot x_i=x_i$, for $i=1,...,p$. Introduce a dummy variable $\tau$ to effect the variation and then integration by parts yields
$$\displaystyle{\frac{\partial}{\partial \tau}\int L(\kappa_1,\kappa_2,...)\mathrm{d}\mathbf{x}=\int\Big[\sum_\alpha\mathsf{E}^\alpha(L)I_\tau^\alpha+\mathsf{Div}(P)\Big]\mathrm{d}\mathbf{x}},$$
where this defines the $p$-tuple $P$, whose components are of the form
$$P_i=\displaystyle{\sum_{\alpha,J} I^\alpha_{\tau J}C^\alpha_{i,J},\qquad i=1,...,p,}$$
and the vectors $\mathcal{C}^\alpha_{i}=(C^\alpha_{i,J})$. Recall that $I^\alpha_{\tau J}=I(u^\alpha_{\tau J})$, where $J$ is an index with respect to the independent variables $x_i$, for $i=1,...,p$. Let $(a_1,...,a_r)$ be coordinates of $G$ near the identity $e$, and $\mathbf{v}_i$, for $i=1,...,r$, the associated infinitesimal vector fields. Furthermore, let $\mathcal{A}d(g)$ be the Adjoint representation of $G$ with respect to these vector fields. For each dependent variable, define the matrix of infinitesimals to be
$$\Omega^\alpha(\widetilde{z})=\left(\widetilde{\zeta^i_j}\right),$$
where 
$$\zeta^i_j=\left.\frac{\partial \widetilde{z_i}}{\partial a_j}\right|_{g=e}$$
are the infinitesimals of the prolonged group action. Let $\Omega^\alpha(I)$, for $\alpha=1,...,q$ be the invariantised version of the above matrices. Then the $r$ conservation laws obtained via Noether's Theorem can be written in the form
$$\sum_i\frac{D}{Dx_i}\mathcal{A}d(\rho)^{-1}\boldsymbol{\upsilon}_i(I)=0,$$
where 
$$\boldsymbol{\upsilon}_i(I)=\sum_\alpha\Omega^\alpha(I)\mathcal{C}^\alpha_i.$$
\end{theorem}

\begin{proof} We denote $\partial u^\alpha/\partial {x_{j}}$ by $u^\alpha_{j}$. We know that 
$$\displaystyle{\left.\frac{\mathrm{d}}{\mathrm{d}\varepsilon}\right|_{\varepsilon=0}\mathscr{L}[u^\alpha+\varepsilon v^\alpha]\;\mathrm{and}\; \left.\frac{\partial}{\partial \tau}\right|_{u^\alpha_\tau=v^\alpha}\mathscr{L}[u^\alpha]}$$
yield the same symbolic result. Thus, 
$$\left.\frac{\partial}{\partial \tau}\right|_{u^\alpha_\tau=v^\alpha}\mathscr{L}[u^\alpha]=0$$
provides us with the following boundary term
\begin{equation}\label{maisuma}
\sum_i\frac{D}{Dx_i}\sum_\alpha\begin{pmatrix}I^\alpha_{\tau j_1} & I^\alpha_{\tau j_1 j_2} & \cdots\end{pmatrix}\mathcal{C}^\alpha_i=0.
\end{equation}
By definition, $I^\alpha_{\tau J}$ is equal to 
$$I^\alpha_{\tau J}=\widetilde{u^\alpha_{\tau J}}|_{g=\rho(z)}.$$
Hence, by the chain rule, 
\begin{equation}\label{maisduas}
\begin{pmatrix}I^{\alpha}_{\tau}&I^{\alpha}_{\tau j_1}&I^{\alpha}_{\tau j_1j_2}&...\end{pmatrix}=\begin{pmatrix}u^\alpha_\tau&u^\alpha_{\tau j_1}&u^\alpha_{\tau j_1 j_2}&\cdots\end{pmatrix}\left.\frac{\partial (\widetilde{u^\alpha},\widetilde{u^\alpha_{j_1}},\widetilde{u^\alpha_{j_1 j_2}},...)}{\partial (u^\alpha,u^\alpha_{j_1},u^\alpha_{j_1 j_2},...)}\right|_{g=\rho(z)}^T,
\end{equation}
where the indices $j_\kappa$ represent the derivatives with respect to any independent variable except the dummy variable $\tau$. We now set
\begin{equation}\label{maistres}
\displaystyle{\left.\frac{\partial \widetilde{u^\alpha}}{\partial \tau}\right|_{g=e}=u^\alpha_\tau=\phi^\alpha_i=\left.\frac{\partial u^\alpha}{\partial a_i}\right|_{g=e}},
\end{equation}
and we know
\begin{equation}\label{maisquatro}
\mathcal{A}d(\rho)^{-1}\Omega(I)=\Omega(z)\left.\frac{\partial \widetilde{z}}{\partial z}\right|_{g=\rho(z)}^T,
\end{equation}
where $\frac{\partial \widetilde{z}}{\partial z}$ corresponds to $\frac{\partial (\widetilde{u^\alpha},\widetilde{u^\alpha_{j_1}},\widetilde{u^\alpha_{j_1 j_2}},...)}{\partial (u^\alpha,u^\alpha_{j_1},u^\alpha_{j_1 j_2},...)}$ (see Theorem 3.3.10 in \cite{Mansfield}).

Substituting the vector $\begin{pmatrix}I^\alpha_{\tau}&I^\alpha_{\tau j_1}&I^\alpha_{\tau j_1j_2}&...\end{pmatrix}$ in (\ref{maisuma}) by its expression in Equation (\ref{maisduas}) yields
$$\sum_i\frac{D}{Dx_i}\sum_\alpha \begin{pmatrix}u^\alpha_\tau&u^\alpha_{\tau j_1}&u^\alpha_{\tau j_1 j_2}&\cdots\end{pmatrix}\left.\frac{\partial \widetilde{z}}{\partial z}\right|_{g=\rho(z)}^T\mathcal{C}^\alpha_i=0.$$
By (\ref{maistres}), the vector $\begin{pmatrix}u^\alpha_\tau&u^\alpha_{\tau j_1}&u^\alpha_{\tau j_1 j_2}&\cdots\end{pmatrix}$ in the above equation can be substituted by every single row of the matrix of infinitesimals $\Omega^\alpha(z)$, as defined in Equation (\ref{infmatgen}). Hence, for each 
independent group parameter $a_j$ we obtain
\begin{equation}
\sum_i\frac{D}{Dx_i}\sum_\alpha\Omega^\alpha_j(z)\left.\frac{\partial \widetilde{z}}{\partial z}\right|_{g=\rho(z)}^T\mathcal{C}^\alpha_i=0,\qquad j=1,...,r,
\end{equation}
where $\Omega^\alpha_j(z)$ corresponds to row $j$ in $\Omega^\alpha(z)$.

If we have $r$ group parameters describing group elements near the identity of the group, we can 
write the $r$ equations in matrix form as
$$\sum_i\frac{D}{Dx_i}\sum_\alpha\Omega^\alpha(z)\left.\frac{\partial\widetilde{z}}{\partial z}\right|_{g=\rho(z)}^T\mathcal{C}^\alpha_i=0.$$
Finally, using Equation (\ref{maisquatro}), we obtain
$$\sum_i\frac{D}{Dx_i}\mathcal{A}d(\rho)^{-1}\sum_\alpha\Omega^\alpha(I)\mathcal{C}^\alpha_i=0.$$
$\hfill \Box$
\end{proof}

If there is only one independent variable, we obtain a set of $r$ first integrals,
\begin{equation}\label{NTfirstintinv} {\bf c}={\cal A}d(\rho)^{-1}\boldsymbol{\upsilon}(I)={\cal A}d(\rho)^{-1}\sum_\alpha\Omega^\alpha(I){\cal C}^\alpha,\end{equation}
where ${\bf c}=\left( c_1\ c_2\ \dots\ c_r\right)^T$. 
Once the invariants are obtained by solving the Euler-Lagrange equations,
this form of the conservation laws can be used to solve for the original dependent variables, as demonstrated in the motivating example of Section \ref{sec:motiv} and in the examples in Section \ref{sec:firstintsl2} provided the vector of constants is not zero. We note that if the $c_i$ are all zero,
the methods detailed in \cite{Mansfield} (\S 7) can be used.

We demonstrate the details in the following examples.

\vspace{0.5cm}
\textbf{Example \ref{sl2surfWITHLags} (cont.)} 
{\it In order to compute conservation laws we specify the order of the invariantised Lagrangian. Thus we will consider
\begin{equation}\label{novoinvL}
\int [L(\sigma,\sigma_x,\sigma_t,\kappa,\kappa_x\kappa_t)+\lambda(x,t)(\mathcal{D}_t\sigma-\mathcal{H}\kappa)]\mathrm{d}x\mathrm{d}t.
\end{equation}
We showed earlier that the terms involving $\lambda$ disappear in the calculation of the Euler-Lagrange equations, and here we show
they disappear in the calculation of the conservation laws.
Recall that $\mathcal{D}_\tau\sigma-\mathcal{H}\kappa=0$, where $\mathcal{H}=\mathcal{D}_x^3+2\sigma\mathcal{D}_x+\sigma_x$, is the syzygy between $\sigma$ and $\kappa$, and that with the introduction of a dummy variable $\tau$ we get a new invariant $I^u_\tau$, and thus an extra two syzygies
\begin{equation}\label{maisseis}
\mathcal{D}_\tau\begin{pmatrix}\sigma\\\kappa\end{pmatrix}=\begin{pmatrix}
\mathcal{D}_x^3+2\sigma\mathcal{D}_x+\sigma_x\\\mathcal{D}_t-\kappa\mathcal{D}_x+\kappa_x\end{pmatrix}I^u_\tau.
\end{equation}
To compute the vectors of invariants $\boldsymbol{\upsilon}_i(I)$ for $i=1,2$, as for the calculation of the Euler-Lagrange equation, we must differentiate (\ref{novoinvL}) with respect to $\tau$ under the integral sign and then integrate by parts in two steps. However, now we will keep track of the boundary terms. Thus after the first set of integration by parts we obtain
$$\begin{array}{l}
\displaystyle{\iint \Big[(\mathsf{E}^\sigma(L)-\lambda_t-\lambda\kappa_x+\lambda_x\kappa)\underbrace{\mathcal{D}_\tau \sigma}+(\mathsf{E}^\kappa(L)+\lambda_{xxx}+\lambda\sigma_x+2\lambda_x\sigma)\underbrace{\mathcal{D}_\tau \kappa}}\\[11pt]
\quad\displaystyle{+\mathcal{D}_x\left(\left(\frac{\partial L}{\partial \sigma_x}-\lambda\kappa\right)\mathcal{D}_\tau \sigma+\left(\frac{\partial L}{\partial \kappa_x}-2\lambda\sigma-\lambda_{xx}\right)\mathcal{D}_\tau \kappa+\lambda_x\mathcal{D}_x\mathcal{D}_\tau \kappa -\lambda\mathcal{D}_x^2\mathcal{D}_\tau \kappa\right)}\\[11pt]
\quad\displaystyle{+\mathcal{D}_t\left(\left(\frac{\partial L}{\partial \sigma_t}+\lambda\right)\mathcal{D}_\tau \sigma+\frac{\partial L}{\partial \kappa_t}\mathcal{D}_\tau \kappa\right)\Big]\mathrm{d}x\mathrm{d}t}.
\end{array}$$
After replacing the underlined syzygies using (\ref{maisseis}), we perform a second set of integration by parts which yields
$$\begin{array}{l}
\displaystyle{\iint \Big[\big((-\mathcal{D}_x^3-2\sigma\mathcal{D}_x-\sigma_x)\mathsf{E}^\sigma(L)+(-\mathcal{D}_t+\kappa\mathcal{D}_x+2\kappa_x)\mathsf{E}^\kappa(L)}\\[11pt]
\displaystyle{+\lambda(\mathcal{H}\kappa-\mathcal{D}_t\sigma)_x+2\lambda_x(\mathcal{H}\kappa-\mathcal{D}_t\sigma)\big)I^u_\tau}\\[11pt]
\displaystyle{+\mathcal{D}_x\Bigg(\Bigg(\frac{\partial L}{\partial \sigma_x}-\lambda\kappa\Bigg)\mathcal{D}_\tau\sigma+\Bigg(\frac{\partial L}{\partial \kappa_x}-2\lambda\sigma-\lambda_{xx}\Bigg)\mathcal{D}_\tau\kappa+\lambda_x\mathcal{D}_x\mathcal{D}_\tau\kappa-\lambda\mathcal{D}_x^2\mathcal{D}_\tau\kappa+\Big(\mathcal{D}_x^2\mathsf{E}^\sigma(L)}\\[11pt]
+2\sigma\mathsf{E}^\sigma(L)-\kappa\mathsf{E}^\kappa(L)-\lambda\kappa\sigma_x-2\lambda\sigma\kappa_x-\lambda\kappa_{xxx}-2\lambda_t\sigma-\lambda_x\kappa_{xx}+\lambda_{xx}\kappa_x-\lambda_{xxt}\Big)I^u_\tau\\[11pt]
+(-\mathcal{D}_x\mathsf{E}^\sigma(L)+\lambda\kappa_{xx}-\lambda_{xx}\kappa+\lambda_{xt})\mathcal{D}_xI^u_\tau+(\mathsf{E}^\sigma(L)-\lambda\kappa_x+\lambda_x\kappa-\lambda_t)\mathcal{D}_x^2I^u_\tau\Bigg)\\[11pt]
\displaystyle{+\mathcal{D}_t\Bigg(\Bigg(\frac{\partial L}{\partial \sigma_t}+\lambda\Bigg)\mathcal{D}_\tau\sigma+\frac{\partial L}{\partial \kappa_t}\mathcal{D}_\tau\kappa+(\mathsf{E}^\kappa(L)+\lambda\sigma_x+2\lambda_x\sigma+\lambda_{xxx})I^u_\tau\Bigg)\Big]\mathrm{d}x\mathrm{d}t.}
\end{array}$$
Due to the relation between $\sigma$ and $\kappa$, it can be checked that
all the terms involving $\lambda$ will disappear. Finally, substituting $\mathcal{D}_\tau\sigma$, $\mathcal{D}_\tau\kappa$, $\mathcal{D}_xI^u_\tau$ and $\mathcal{D}_x^2I^u_\tau$ respectively by the following invariant differential formulae 
$$\begin{array}{l}
\mathcal{D}_\tau \sigma=I^u_{111\tau}-\sigma I^u_{1\tau},\\
\mathcal{D}_\tau \kappa=I^u_{2\tau}-\kappa I^u_{1\tau},\\
\mathcal{D}_xI^u_\tau=I^u_{1\tau},\\
\mathcal{D}_x^2I^u_\tau=I^u_{11\tau}-\sigma I^u_\tau,
\end{array}$$
yields the boundary terms
$$\begin{array}{l}
\displaystyle{\mathcal{D}_x\left((\sigma\mathsf{E}^\sigma(L)-\kappa\mathsf{E}^\kappa(L)+\mathcal{D}_x^2\mathsf{E}^\sigma(L))I^u_\tau+\left(-\frac{\partial L}{\partial \kappa_x}\kappa-\frac{\partial L}{\partial \sigma_x}\sigma-\mathcal{D}_x\mathsf{E}^\sigma(L)\right)I^u_{1\tau}\right.}\\[11pt]
\displaystyle{\left.+\mathsf{E}^\sigma(L)I^u_{11\tau}+\frac{\partial L}{\partial \sigma_x}I^u_{111\tau}+\frac{\partial L}{\partial\kappa_x}I^u_{2\tau}\right)}\\[11pt]
\displaystyle{+\mathcal{D}_t\left(\mathsf{E}^\kappa(L)I^u_\tau+\left(-\frac{\partial L}{\partial \sigma_t}\sigma-\frac{\partial L}{\partial \kappa_t}\kappa\right)I^u_{1\tau}+\frac{\partial L}{\partial \sigma_t}I^u_{111\tau}+\frac{\partial L}{\partial \kappa_t}I^u_{2\tau}\right).}
\end{array}$$
Next, using the matrix of invariantised infinitesimals below
$$\Omega^u(I)=\begin{pmatrix}
0 & 2 & 0 & 2\sigma & 2\kappa\\
1 & 0 & 0 & 0 & 0\\
0 & 0 & -2 & 0 & 0
\end{pmatrix},$$
we get the vectors of invariants
$$\boldsymbol{\upsilon}_1(I)=\begin{pmatrix}
-2\mathcal{D}_x\mathsf{E}^\sigma(L)\\
\sigma\mathsf{E}^\sigma(L)-\kappa\mathsf{E}^\kappa(L)+\mathcal{D}_x^2\mathsf{E}^\sigma(L)\\
-2\mathsf{E}^\sigma(L)
\end{pmatrix},\qquad \boldsymbol{\upsilon}_2(I)=
\begin{pmatrix}
0\\
\mathsf{E}^\kappa(L)\\
0
\end{pmatrix}.$$

Finally inverting $\mathcal{A}d(g)$ found in Example \ref{SL2cont} and evaluating it at the frame (\ref{ProjSL2frame}) gives 
$$\mathcal{A}d(\rho)^{-1}=\begin{pmatrix}
\displaystyle{1-\frac{uu_{xx}}{u_x^2}} & \displaystyle{\frac{2u}{u_x}} & \displaystyle{\frac{u_{xx}}{u_x}-\frac{uu_{xx}^2}{2u_x^3}}\\[10pt]
\displaystyle{-\frac{u_{xx}}{2u_x^2}} & \displaystyle{\frac{1}{u_x}} & \displaystyle{-\frac{u_{xx}^2}{4u_x^3}}\\[10pt]
\displaystyle{-u+\frac{u^2u_{xx}}{2u_x^2}} & \displaystyle{-\frac{u^2}{u_x}} & \displaystyle{u_x-\frac{uu_{xx}}{u_x}+\frac{u^2u_{xx}^2}{4u_x^3}}
\end{pmatrix}.$$

Hence, the conservation laws are
\begin{equation}\label{ultimaconserv}
\mathcal{D}_x\left(\mathcal{A}d(\rho)^{-1}\begin{pmatrix}
\displaystyle{-2\frac{\mathrm{d}}{\mathrm{d}x}\mathsf{E}^\sigma(L)}\\[11pt]
\displaystyle{\sigma\mathsf{E}^\sigma(L)-\kappa\mathsf{E}^\kappa(L)+\frac{\mathrm{d}^2}{\mathrm{d}x^2}\mathsf{E}^\sigma(L)}\\[11pt]
-2\mathsf{E}^\sigma(L)
\end{pmatrix}\right)+\mathcal{D}_t\left(\mathcal{A}d(\rho)^{-1}\begin{pmatrix}
0\\[10pt]
\mathsf{E}^\kappa(L)\\[10pt]
0
\end{pmatrix}\right)=0.
\end{equation}
}

\begin{remark} Equation (\ref{ultimaconserv}) shows the structure of the conservation laws much more clearly than lengthy expressions in the original variables. It is in this sense that our theorem ``adds value'' to Noether's result.
\end{remark}

\section{Conservation laws from Semisimple Groups}\label{ultimo}
The result in Theorem \ref{tmel} gives the conservation laws for a variational problem in a particular form which we can use to advantage.
In this section, we consider one dimensional problems in the case that the group is semisimple, the case in which the Killing form $\mathbf{B}$  is invertible.
We show that one can then always obtain a first integral of the Euler-Lagrange system. We then examine all three inequivalent actions of $SL(2)$
in the plane, and show how our expression of the conservation laws can be used to radically simplify the integration problem in each case.

\begin{theorem}\label{simplificacao}
Consider $\mathbf{v}\in \mathcal{X}_G(M)$, where $\mathcal{X}_G(M)$ is a semisimple Lie algebra of infinitesimal vector fields that generate the transformation group $G$. Let $\mathbf{B}$ be the Killing form for $\mathcal{X}_G(M)$. Let $L(\kappa^\alpha,\kappa_s^\alpha,...)\mathrm{d}s$ be invariant under the group action of $G$, which leaves the only independent variable $s$ unchanged. Then
$$\boldsymbol{\upsilon}(I)^T\mathbf{B}^{-1}\boldsymbol{\upsilon}(I)=\mathbf{c}^T\mathbf{B}^{-1}\mathbf{c}$$
is a first integral for the Euler-Lagrange equations $\mathsf{E}^\alpha(L)=0$, for $\alpha=1,...,q$, where $\boldsymbol{\upsilon}(I)$ is given in Theorem \ref{tmel} and $\mathbf{c}$ is a constant vector.
\end{theorem}

\begin{proof}
From Theorem \ref{tmel} we know that $\mathcal{A}d(\rho)^{-1}\boldsymbol{\upsilon}(I)=\mathbf{c}$. Since $\mathcal{X}_G(M)$ is semisimple, we can multiply both sides by $\mathbf{c}^T\mathbf{B}^{-1}$ and obtain
$$\mathbf{c}^T\mathbf{B}^{-1}\mathcal{A}d(\rho)^{-1}\boldsymbol{\upsilon}(I)=\mathbf{c}^T\mathbf{B}^{-1}\mathbf{c}.$$
Substituting the vector $\mathbf{c}^T$ by $\boldsymbol{\upsilon}(I)^T\mathcal{A}d(\rho)^{-T}$ on the left-hand side gives us 
\begin{equation}\label{quase}
\boldsymbol{\upsilon}(I)^T\mathcal{A}d(\rho)^{-T}\mathbf{B}^{-1}\mathcal{A}d(\rho)^{-1}\boldsymbol{\upsilon}(I)=\mathbf{c}^T\mathbf{B}^{-1}\mathbf{c}.
\end{equation}
Using Equation (\ref{BAdinveq}), i.e. $\mathbf{B}=\mathcal{A}d_g\mathbf{B}\mathcal{A}d_g^{\,T}$, we can simplify Equation (\ref{quase}) which yields the result. $\hfill \Box$
\end{proof}

Looking again at the equality $\mathcal{A}d(\rho)^{-1}\boldsymbol{\upsilon}(I)=\mathbf{c}$, multiplying both sides of it by $\mathbf{B}^{-1}$
and then using Equation (\ref{BAdinveq}), we obtain
$$\mathcal{A}d(\rho)^T\mathbf{B}^{-1}\boldsymbol{\upsilon}(I)=\mathbf{B}^{-1}\mathbf{c}.$$
In the examples that follow, we write the conservation laws in the form
\begin{equation}\label{remarksimp}
\Omega(z)^T\mathcal{A}d(\rho)^T\mathbf{B}^{-1}\boldsymbol{\upsilon}(I)=\Omega(z)^T\mathbf{B}^{-1}\mathbf{c},\end{equation}
which yields a remarkable simplification in the system to be solved.

\subsection{Integration Results for the $\mathbf{SL(2)}$ Actions on the Plane}\label{sec:firstintsl2}
In this section, we will calculate the conservation laws associated to variational problems that are invariant under the three inequivalent $SL(2,\mathbb{C})$ actions and find the solutions that extremise these variational problems.

We assume the vector $\mathbf{c}$ of constants is non-zero. 

So taking the coordinates of $\mathbb{C}^2$ to be $(x,u)$ and a generic element of $SL(2,\mathbb{C})$ to be
$$g=\begin{pmatrix}
a & b\\
c & d\end{pmatrix},$$
where $ad-bc=1$, then the three inequivalent actions are:
\begin{itemize}
\item[]\textbf{Action 1}
$$\displaystyle{\widetilde{x}=x,\qquad\widetilde{u}=\frac{au+b}{cu+d}},$$
\item[]\textbf{Action 2}
$$\displaystyle{\widetilde{x}=\frac{ax+b}{cx+d},\qquad \widetilde{u}=\frac{u}{(cx+d)^2}},$$
\item[]\textbf{Action 3}
$$\displaystyle{\widetilde{x}=\frac{ax+b}{cx+d},\qquad \widetilde{u}=6c(cx+d)+(cx+d)^2u}.$$
\end{itemize}

\subsubsection{$\mathbf{SL(2)}$ Action 1}
Consider the variational problem $\int L(\sigma,\sigma_s)\mathrm{d}s$ invariant under the $SL(2)$ Action $1$, with a frame defined by the normalisation equations 
$$\widetilde{u}=0,\qquad \widetilde{u_s}=1,\qquad\mathrm{and}\qquad \widetilde{u_{ss}}=0,$$
and the generating differential invariant $\{u;s\}=I^u_{111}=\sigma$. Then the Euler-Lagrange equation is 
$$\mathsf{E}^u(L)=(-\mathcal{D}_x^3-2\sigma\mathcal{D}_x-\sigma_x)\mathsf{E}^\sigma(L)=0$$
and the conservation laws  are
$$\begin{pmatrix}
\displaystyle{1-\frac{uu_{ss}}{u_s^2}} & \displaystyle{\frac{2u}{u_s}} & \displaystyle{\frac{u_{ss}}{u_s}-\frac{uu_{ss}^2}{2u_s^3}}\\[10pt]
\displaystyle{-\frac{u_{ss}}{2u_s^2}} & \displaystyle{\frac{1}{u_s}} & \displaystyle{-\frac{u_{ss}^2}{4u_s^3}}\\[10pt]
\displaystyle{-u+\frac{u^2u_{ss}}{2u_s^2}} & \displaystyle{-\frac{u^2}{u_s}} & \displaystyle{u_s-\frac{uu_{ss}}{u_s}+\frac{u^2u_{ss}^2}{4u_s^3}}
\end{pmatrix}\begin{pmatrix}
\displaystyle{-2\mathcal{D}_s\mathsf{E}^\sigma(L)}\\[10pt]
\displaystyle{\sigma\mathsf{E}^\sigma(L)+\mathcal{D}_s^2\mathsf{E}^\sigma(L)}\\[10pt]
-2\mathsf{E}^\sigma(L)
\end{pmatrix}=\begin{pmatrix}
c_1\\[10pt]c_2\\[10pt]c_3
\end{pmatrix}.$$

Now using  Theorem \ref{simplificacao}, where $\mathbf{B}$ is as in (\ref{killingmatrix}), we obtain the first integral of the Euler-Lagrange equation $\mathsf{E}^u(L)$,
$$4(\mathcal{D}_s\mathsf{E}^\sigma(L))^2-8\mathsf{E}^\sigma(L)\mathcal{D}_s^2\mathsf{E}^\sigma(L)-8\sigma(\mathsf{E}^\sigma(L))^2=c_1^2+4c_2c_3.$$

Next, rewriting the conservation laws in the form (\ref{remarksimp}), we obtain a simplified system containing the following equation 
\begin{eqnarray}
-2\mathsf{E}^\sigma(L)u_s-c_1u+c_2u^2-c_3=0.\label{systema11}
\end{eqnarray}
Equation (\ref{systema11}) is a first order ODE. It can be transformed into a Riccati equation with constant coefficients by setting $\tau=\int \frac{1}{2\mathsf{E}^\sigma(L)}\mathrm{d}s$, yielding
$$u_\tau=-c_1u+c_2u^2-c_3.$$
Thus, once we have solved for $\sigma$, the solution of Equation (\ref{systema11}) is
$$u(s)=\frac{c_1}{2c_2}-\frac{\beta}{2c_2}\tanh\left(\frac{1}{2}\beta f(s)\right),$$
where $\beta=\sqrt{c_2^2+4c_2c_3}$ and $f(s)=\int\frac{1}{2\mathsf{E}^\sigma(L)}\mathrm{d}s+c_4$.
We note the remaining equations coming from the conservation laws all then simplify to zero.

\subsubsection{$\mathbf{SL(2)}$ Action 2}
In this case, we reparametrise $(x, u(x))$ as $(x(s), u(s))$ and we may take an additional equation to fix the parametrisation,
provided the result leads to the full solution set. By construction, $s$ is invariant and thus $\mathcal{D}_s={\rm d}/{\rm d}s$.
We do this to simplify the calculation of the conservation laws, as it sends denominators to unity.
We take the frame for the $SL(2)$ Action 2 to be defined by the normalisation equations 
$$\widetilde{x}=0,\qquad \widetilde{u}=1,\qquad\mathrm{and}\qquad \widetilde{u_s}=0,$$
and the generating invariants are $I^x_1$ and $I^u_{11}$, which we will rename as $\eta$ and $\sigma$, respectively. 
The additional equation we take is $\eta=1$ and this is introduced as a constraint. Thus we consider the variational problem $\int [L(\sigma,\sigma_s,\sigma_{ss})-\lambda(s)(\eta-1)]\mathrm{d}s$
where $\lambda$ is the Lagrange multiplier. After using $\mathsf{E}^x(L)=0$ to eliminate $\lambda$, we obtain
$$\begin{array}{l}
\displaystyle{\mathsf{E}^u(L)=\mathcal{D}_s^2\mathsf{E}^\sigma(L)-2\sigma\mathsf{E}^\sigma(L)+L-\left(\frac{\partial L}{\partial \sigma_s}-\mathcal{D}_s\left(\frac{\partial L}{\partial \sigma_{ss}}\right)\right)\sigma_s+\frac{\partial L}{\partial \sigma_{ss}}\sigma_{ss}=0,}
\end{array}$$
and the conservation laws are
$$\begin{pmatrix}
\displaystyle{1-\frac{xu_s}{ux_s}} & \displaystyle{\frac{2x}{u}} & \displaystyle{\frac{u_s}{x_s}-\frac{xu_s^2}{2ux_s^2}}\\[10pt]
\displaystyle{-\frac{u_s}{2ux_s}} & \displaystyle{\frac{1}{u}} & \displaystyle{-\frac{u_s^2}{4ux_s^2}}\\[10pt]
\displaystyle{-x+\frac{x^2u_s}{2ux_s}} & \displaystyle{-\frac{x^2}{u}} & \displaystyle{u-\frac{xu_s}{x_s}+\frac{x^2u_s^2}{4ux_s^2}}\end{pmatrix}\boldsymbol{\upsilon}(I)=\begin{pmatrix}c_1\\c_2\\c_3\end{pmatrix},$$
where the vector of invariants $\boldsymbol{\upsilon}(I)$ is 
$$\boldsymbol{\upsilon}(I)=\begin{pmatrix}-2\mathcal{D}_s\mathsf{E}^\sigma(L)\\[10pt]
\mathsf{E}^\sigma(L)-2\sigma\mathsf{E}^\sigma(L)+L-\left(\frac{\partial L}{\partial \sigma_s}-\mathcal{D}_s\left(\frac{\partial L}{\partial \sigma_{ss}}\right)\right)\sigma_s-\frac{\partial L}{\partial \sigma_{ss}}\sigma_{ss}\\[10pt]
-2\mathsf{E}^\sigma(L)\end{pmatrix}.$$
Applying Theorem \ref{simplificacao}, we obtain the first integral of the Euler-Lagrange equation to be
\begin{eqnarray}
4(\mathcal{D}_s\mathsf{E}^\sigma(L))^2-8\mathsf{E}^\sigma(L)\Bigg(\mathsf{E}^\sigma(L)(1-2\sigma)+L\qquad\qquad\qquad\nonumber\\[10pt]
\left.-\left(\frac{\partial L}{\partial \sigma_s}-\mathcal{D}_s\left(\frac{\partial L}{\partial \sigma_{ss}}\right)\right)\sigma_s-\frac{\partial L}{\partial \sigma_{ss}}\sigma_{ss}\right)=c_1^2+4c_2c_3.\label{fiaction2}
\end{eqnarray}
Next, writing the laws in the form (\ref{remarksimp}), we obtain  the equation, 
\begin{eqnarray}
-2\mathsf{E}^\sigma(L)u-c_1x+c_2x^2-c_3=0.\label{systema21} 
\end{eqnarray}
Assuming we have solved the Euler-Lagrange equation for $\sigma$, we can solve this equation together with the constraint $\eta=1$ for $x$ and $u$. Recall that $\eta=I^x_1=\widetilde{x_s}|_{frame}=\frac{x_s}{u}=1$, thus $x_s=u$. Hence Equation (\ref{systema21}) becomes
$$-2\mathsf{E}^\sigma(L)x_s-c_1x+c_2x^2-c_3=0,$$
which is the same equation as Equation (\ref{systema11}). Thus, the solution for $x$ is 
$$x(s)=\frac{c_1}{2c_2}-\frac{\beta}{2c_2}\tanh\left(\frac{1}{2}\beta f(s)\right),$$
where $\beta=\sqrt{c_1^2+4c_2c_3}$ and $f(s)=\int\frac{1}{2\mathsf{E}^\sigma(L)}\mathrm{d}s+c_4$. Differentiating this with respect to $s$ will give us the solution for $u$;
$$u(s)=-\frac{\beta^2}{8c_2\mathsf{E}^\sigma(L)}\,\mathrm{sech}^2\left(\frac{1}{2}\beta f(s)\right).$$
Note that the restriction on $\eta$ does not lead to a reduction in the number of independent constants in the solution.
We note the remaining equations coming from the conservation laws all then simplify to zero.

\subsubsection{$\mathbf{SL(2)}$ Action 3}
Again, we reparametrise $(x, u(x))$ as $(x(s), u(s))$ and we may take an additional equation to fix the parametrisation,
provided the result leads to the full solution set. By construction, $s$ is invariant and thus $\mathcal{D}_s={\rm d}/{\rm d}s$,
which dramatically simplifies the calculations.
For a Lagrangian $L(\eta,\eta_s,\sigma,\sigma_s,\sigma_{ss})\mathrm{d}s$ invariant under the $SL(2)$ Action $3$, whose moving frame is defined by the normalisation equations 
$$\widetilde{x}=0,\qquad \widetilde{x_s}=1,\qquad\mathrm{and}\qquad \widetilde{u}=0,$$
and for which the set of generating invariants is $\{I^x_{11}=\eta,I^u_1=\sigma\}$, the Euler-Lagrange equations are 
$$\mathsf{E}^x(L)=\mathcal{D}_s^2\mathsf{E}^\eta(L)-\eta\mathcal{D}_s\mathsf{E}^\eta(L)-\frac{1}{3}\sigma\mathsf{E}^\eta(L)-\sigma\mathcal{D}_s\mathsf{E}^\sigma(L)+\eta\sigma\mathsf{E}^\sigma(L)-\sigma_s\mathsf{E}\sigma(L)=0,$$
$$\mathsf{E}^u(L)=\frac{1}{3}\mathsf{E}^\eta(L)-\mathcal{D}_s\mathsf{E}^\sigma(L)-\eta\mathsf{E}^\sigma(L)=0,$$
and their associated conservation laws are
\begin{equation}\label{quase4}
\begin{pmatrix}
\displaystyle{1+\frac{1}{3}xu} & \displaystyle{\frac{2x}{x_s}} & \displaystyle{-\frac{1}{18}xx_su^2-\frac{1}{3}x_su}\\[10pt]
\displaystyle{\frac{1}{6}u} & \displaystyle{\frac{1}{x_s}} & \displaystyle{-\frac{1}{36}x_su^2}\\[10pt]
\displaystyle{-x-\frac{1}{6}x^2u} & \displaystyle{-\frac{x^2}{x_s}} & \displaystyle{x_s+\frac{1}{3}xx_su+\frac{1}{36}x^2x_su^2}
\end{pmatrix}\begin{pmatrix} 2\mathsf{E}^\eta(L)\\[10pt]
-\mathcal{D}_s\mathsf{E}^\eta(L)+\sigma\mathsf{E}^\sigma(L)\\[10pt]
6\mathsf{E}^\sigma(L)\end{pmatrix}=
\begin{pmatrix}c_1\\[10pt]c_2\\[10pt]c_3
\end{pmatrix}.
\end{equation}

Applying Theorem \ref{simplificacao} to (\ref{quase4}) yields the following first integral for the Euler-Lagrange equations
$$4(\mathsf{E}^\eta(L))^2+24\sigma(\mathsf{E}^\sigma(L))^2-24\mathsf{E}^\sigma(L)\mathcal{D}_s\mathsf{E}^\eta(L)=c_1^2+4c_2c_3.$$

Next writing the conservation laws (\ref{quase4}) using (\ref{remarksimp}) we obtain two equations
\begin{eqnarray}\label{quase5}
6\mathsf{E}^\sigma(L)x_s-c_1x+c_2x^2-c_3=0,\\[10pt]\label{quase6}
2\mathsf{E}^\eta(L)x_s-2\mathsf{E}^\sigma(L)x_s^2u-c_1x_s+2c_2xx_s=0.
\end{eqnarray}
We assume we have first solved the Euler-Lagrange equations for $\sigma$ and $\eta$. Equation (\ref{quase5}) is a first order ODE which can be transformed into a Riccati equation with constant coefficients by setting $\tau=\int \frac{1}{6\mathsf{E}^\sigma(L)}\mathrm{d}s$. Thus, the solution to Equation (\ref{quase5}) is
$$x(s)=\frac{c_1}{2c_2}+\frac{\beta}{2c_2}\tanh\left(\frac{1}{2}\beta f(s)\right),$$
where $\beta=\sqrt{c_1^2+4c_2c_3}$ and $f(s)=\int \frac{1}{6\mathsf{E}^\sigma(L)}\mathrm{d}s+c_4$.
Now simplifying Equation (\ref{quase6}) yields 
$$6\mathsf{E}^\eta(L)-3c_1+6c_2x+u(-c_1x+c_2x^2-c_3)=0,$$
which is a linear equation for $u$. Hence,
$$u(s)=\frac{3c_1-6c_2x-6\mathsf{E}^\eta(L)}{-c_1x+c_2x^2-c_3}.$$
We note the remaining equation coming from the conservation laws simplifies to zero.

\section{Conclusion}
Noether's Theorem is a classical result giving conservation laws for Lie group invariant variational problems. Expressed in the original variables, the conservation laws for high order Lagrangians can have tens of terms which are difficult to analyse. In Theorem $7.4.1$ of \cite{Mansfield}, it is shown that for one dimensional variational problems the essential structure of the laws associated to these can be written in terms of differential invariants and a moving frame. In this paper we have generalised this result to higher dimensional variational problems. In this condensed view, the information contained in the laws becomes clearer.

The laws for one dimensional $SL(2)$ invariant Lagrangians are studied in detail, and we show that in the three inequivalent cases, our methods lead to a far simpler integration problem than that for in the original variables. In \cite{GoncalvesMansfieldI}, we will show the results for $SE(2)$ and $SE(3)$ invariant Lagrangians.


\begin{thebibliography}{99}
\bibitem{KoganOlver}
I. A. KOGAN and P. J. OLVER. Invariant Euler-Lagrange equations and the invariant variational bicomplex, Acta Appl. Math. 76:137--193 (2003).
\bibitem{Mansfield}
E. L. MANSFIELD, A Practical Guide to the Invariant Calculus, Cambridge University Press, Cambridge, 2010.
\bibitem{Noether}
E. NOETHER, Invariante Varlationsprobleme. Nachr. Ges. Wiss. G\"ottingem, Math.-Phys. Kl. 235--257 (1918). An english translation is available at \texttt{arXiv:physics/0503066v1 [physics.hist-ph]}.
\bibitem{Olver}
P. J. OLVER, Applications of Lie Groups to Differential Equations, Second Edition, Springer, New York, 1993.
\bibitem{FelsOlver}
M. FELS and P. J. OLVER. Moving coframes I, Acta Appl. Math. 51:161-312 (1998). 
\bibitem{FelsOlverII}
M. FELS and P. J. OLVER. Moving coframes II, Acta Appl. Math. 55:127--208 (1999).
\bibitem{HubertAA}
E. HUBERT.  Differential algebra for derivations with nontrivial commutation rules, J. Pure Appl. Algebra 200(1-2):163--190 (2005).
\bibitem{HubertAC}
E. HUBERT.  Differential invariants of a Lie group action: syzygies on a generating set, J. Symbolic Comput. 44(4):382--416 (2009a).
\bibitem{HubertAD}
E. HUBERT. Generation properties of Maurer-Cartan invariants. Preprint [hal:inria-00194528] (2009b).
\bibitem{HubertA}
E. HUBERT and I. A. KOGAN. Smooth and algebraic invariants of a group action. Local and Global Constructions, Foundations of Comput. Math. 7(4):345--383 (2007a).
\bibitem{HubertB}
E. HUBERT and I. A. KOGAN. Rational invariants of a group action. Construction and rewriting, J. Symbolic Comput. 42(1-2):203--217 (2007b).
\bibitem{AIDA}
E. HUBERT, AIDA Maple package: Algebraic Invariants and their Differential Algebras, 2007.
\bibitem{Hirsch}
M. W. HIRSCH, Differential Topology, Springer, New York, 1976.
\bibitem{GoncalvesMansfieldI}
T. M. N. GON\c CALVES and E. L. MANSFIELD. Moving frames  and conservation laws for Euclidean invariant Lagrangians, in preparation.
\bibitem{ClarksonOlver}
P. A. CLARKSON and P. J. OLVER. Symmetry of the Chazy equation, J. Differential Equations 124(1):225--246 (1996).

\end{thebibliography}
\end{document}